\documentclass[11pt]{article}
\usepackage{amsmath, amssymb, amscd, amsthm, amsfonts}
\usepackage{graphicx}
\usepackage{hyperref}
\usepackage{mathrsfs}

\title{ Generalized probabilistic approximation characteristic based on Birkhoff orthogonality and related conclusions in $S_{\infty}$-norm  }
\author{Weiye Zhang \and Chong Wang \and Huan Li *}
\date{}

\newtheorem{Definition}{Definition}
\newtheorem{Remark}{Remark}
\newtheorem{thm}{Theorem}[section]
\newtheorem{lem}{Lemma}[section]

\begin{document}

\maketitle

\begin{abstract}
In this article, we generalize the definition of the probabilistic Gel'fand width from the Hilbert space to the strictly convex reflexive space by giving Birkhoff left orthogonal decomposition theorem. Meanwhile, a more natural definition of Gel'fand width in the classical setting is selected to make sure probabilistic and average Gel'fand widths will not lose their meaning, so that we can give the equality relation between the probabilistic Gel'fand width and the probabilistic linear width of the Hilbert space. Meanwhile, we use this relationship to continue the study of the Gel'fand widths of the univariate Sobolev space and the multivariate Sobolev space, especially in $S_{\infty}$-norm, and determine the exact order of probabilistic and average Gel'fand widths.
\end{abstract}


\section{Introduction}

The approximation theory of functions is one of the most important analytical directions in mathematics with a long history. In the current era, it is not only closely related to functional analysis and other branches of mathematics such as numerical analysis, partial differential equations(PDEs), geometric analysis, and so on, but also widely used in information theory and computational complexity, signal processing, image recognition, and even in large language models(LLM). These correlations are particularly prominent in the width theory of approximation theory. The width theory gives the asymptotic order of error when one linear subspace approximates its background linear space. This provides us with a way to theoretically explore the concept of infinity, and it can also be used to handle problems with a huge amount of information.
In terms of computational complexity, Traub, Wasilkowski, and Wozniakowski \cite{ref1} gave the relationship between width and algorithm complexity, which has established relations between the Gel'fand $n$-widths and $n$th minimal radii of information in the classical setting. In this way, we can use the width theory for data processing and algorithm evaluation, especially in LLM.
In the aspect of PDEs, Wenzel et al. \cite{ref2} introduced a scale of greedy selection criteria when studying meshless approximation for solutions of boundary value problems of elliptic partial differential equations. Then they provided bounds on the approximation error for greedy selection criteria and analyzed the corresponding convergence rates by analyzing Kolmogorov widths of special sets of point-evaluation functionals. 

Three kinds of widths, Kolmogorov width, linear width, and Gel'fand width, are the most important concepts in width theory. With the deepening of the research, in order to meet different requirements, two different frameworks, namely probability and average framework, have been proposed on the basis of the classical setting. In recent years, many works have been done in the probabilistic and average setting. 
Maiorov \cite{ref3} studied the Kolmogorov width of Wiener space with $L_\infty$ norm in average setting. Soon, he \cite{ref4,ref5} defined the Kolmogorov width and linear width in probabilistic setting and was the first to use the discretization method to estimate the asymptotic order of the width in probabilistic setting. This method has been used until now. 
Later, Fang and Ye \cite{ref6} used results of Maiorov to calculate the probabilistic and average linear widths of univariate Sobolev space $W_{2}^{r}$ with Gaussian measure in $L_q(1\leq q< \infty)$ norm. And then they \cite{ref7} studied probabilistic and average linear widths of the univariate Sobolev space in $L_\infty$ norm. It is the first time to solve the problem when $q=\infty$.
Chen and Fang \cite{ref8} presented sharp orders of Kolmogorov probabilistic $(n,\delta)\text{-width}$ and $p\text{-average}$ $n\text{-width}$ of multivariate Sobolev space $MW_{2}^{r}$ with Gaussian measure.
Wang \cite{ref9} discussed the widths of weighted Sobolev spaces with the weight $(1-\|x\|_2^2)^{\mu-\frac{1}{2}}$ on the Euclidean unit ball. And he calculated the asymptotic orders of probabilistic linear $(n,\delta)\text{-width}$ and $p\text{-average}$ linear $n\text{-width}$.
Tan and Shao et al. \cite{ref10} proposed the Gel'fand width in the probabilistic setting and estimated the sharp order of the probabilistic Gel'fand width of finite-dimensional space. By using the result of finite-dimensional space, they obtained  the probabilistic Gel'fand width of $W_{2}^{r}$ with Gaussian measure in $L_q(1\leq q< \infty)$ norm. Here, the result in $L_\infty$ was not given.
Liu et al.\cite{ref11} defined the Gel'fand width in the average setting. Meanwhile, he obtained the average Gel'fand widths of $W_{2}^{r}$ and $MW_{2}^{r}$ with Gaussian measure in $L_q(1\leq q< \infty)$ norm. The result here did not include $L_\infty$ either. 
Xu et al.\cite{ref12} determined the asymptotic order of probabilistic and average linear widths of $W_{2}^{r}$ in $S_q(1\leq q\leq \infty)$. 
Wu et al.\cite{ref13} calculated the probabilistic and average Gel'fand widths of $W_{2}^{r}$ and $MW_{2}^{r}$ with Gaussian measure in $S_q(1\leq q< \infty)$ norm. So, the result in $S_{\infty}$ remained.

The innovations of this article lie in two aspects. On one hand, we use the Birkhoff orthogonality to generalize the definition of the probabilistic Gel'fand width from the Hilbert space to the strictly convex reflexive space. In this way, we can consider the probabilistic Gel'fand widths of $L^p$, $l^p$, $W^{k,p}$ where $1<p<\infty,p\neq2$ and so on. Since orthogonal decomposition has never been considered in the general normed space, this promotion has never been carried out before.
On the other hand, we give the Gel'fand widths of the univariate Sobolev space and the multivariate Sobolev space in $S_{\infty}$-norm. For one thing, when $q=\infty$, the properties of $L_q$ spaces and $S_q$ spaces differ fundamentally from those when $1 \leq q < \infty$. From Ref\cite{ref7}, we can clearly find the differences between the proof of $L_\infty$ and those when $1 \leq q < \infty$. For another, the $S_q$ spaces derive from Stiefel manifold. In conclusion, our study in $S_\infty$ is meaningful. We improve the proof of Ref\cite{ref11} and then give the equality relation between the probabilistic Gel'fand width and the probabilistic linear width of the Hilbert space. Meanwhile, we use the relationship mentioned before to continue the study of the Gel'fand widths and determine the exact order of probabilistic and average Gel'fand widths.

In this paper, we choose a better definition of Gel'fand width in the classical setting to make sure probabilistic and average Gel'fand widths will not lose their meaning. 

Firstly, let us review the basics of the width theory. We will first introduce the definitions of Kolmogorov width, linear width, and Gel’fand width in the classical setting. Let $X$ be a normed linear space with norm $\|\cdot\|_X$. $\{0\}\subset W$ is a non-empty subset of $X$. $F_n$ is a linear subspace of $X$ with dimension n, and n is a non-negative integer. We will use these letters throughout the paper. The best approximation of $F_n$ to $W$ in $X$ is defined as 
$$ e(W, F_n, X) := \sup_{x\in W} \inf_{y\in F_n} \|x-y\|. $$
We can clearly find that $e(W, F_n, X)$ represents the deviation or error when the set $W$ is approximated by $F_n$ in the normed space $X$. Figuratively, we can understand it as the "maximum distance" from $F_n$ to $W$ considering the norm $\|\cdot\|_X$. It is similar to the Gromov-Hausdorff distance.

Then the quantity
$$ d_n(W, X) := \inf_{F_n} e(W, F_n, X)=\inf_{F_n} \sup_{x\in W} \inf_{y\in F_n} \|x-y\| $$
is called the Kolmogorov $n$-width of $W$ in $X$. In the first infimum, we take over all $n$-dimensional subspaces $F_n$ of $X$. 

Let $T_n$ be a bounded linear operator on $X$ with rank at most $n$ and let $F^n$ be a closed linear subspace of $X$ with co-dimension n. The definitions of the linear $n$-width and the Gel’fand $n$-width are as follows:
\begin{align*}
    \lambda_{n}(W,X) &:= \inf_{T_{n}} \sup_{x\in W} \|x-T_{n}x\| \\
    d^{n}(W,X) &:= \inf_{F^{n}} \sup_{x\in W\cap F^{n}}\|x\| 
\end{align*}
where the infimum in $\lambda_{n}(W,X)$ is taken over all $T_n$ on $X$ and the infimum in $ d^{n}(W,X)$ is taken over all $F^n$ on $X$. In particular, since $\{0\}\subset W$ and $\{0\}\subset F^n$, $W\cap F^{n}$ is not empty. More elementary information about these three widths can be obtained in the book \cite{ref14}.

Next, we will consider widths in the probabilistic and average setting.

\begin{Definition}[\cite{ref13}]\label{def1}
    Let $X$ and $T^n$ be consistent with $\lambda_n(W,X)$. Let $\mathscr{B}$ be the Borel field on $W$, which is the smallest $\sigma\text{-algebra}$ consisting of all open subsets of $W$, and let $\mu$ be the probabilistic measure defined on $\mathscr{B}$. For any $\delta\in \left[ 0,1 \right]$, the probabilistic linear $(n,\delta)$-width of $W$ in $X$ is given by
	\begin{align*}
		&\lambda_{n,\delta}(W,\mu,X):=\inf_{G_\delta}\lambda_{n}(W\setminus G_{\delta},X)=\inf_{G_\delta}\inf_{T_{n}} \sup_{x\in (W\setminus G_\delta)} \|x-T_{n}x\|,
	\end{align*}
    where $G_{\delta}$ runs over all possible subsets in $\mathscr{B}$ with $\mu(G_{\delta})\leq \delta$. 
    
    The $p$-average linear $n$-width of $W$ in $X$ is given by
	\begin{align*}
		\lambda_{n}^{(a)}(W,\mu,X)_{p}:=\inf_{T_{n}}\left(\int_{W}\|x-T_{n}x\|^{p}d\mu(x)\right)^{1/p},\ 0<p<\infty.
	\end{align*}
\end{Definition}

\begin{Remark}\label{rmk1}
    In particular, when $\delta = 0$ and $G_\delta = \varnothing$, the probabilistic linear $(n,\delta)$-width is the linear width in the classical setting. When $\delta = 1$, according to the definition of the linear $n$-width, $G_\delta \neq W$ and $W\setminus G_\delta$ is a non-empty set of measure zero.
\end{Remark}

Then, let us recall how the probabilistic Gel'fand width was introduced according to Ref\cite{ref10}. Firstly, let $H$ be a Hilbert space equipped with the probabilistic measure $\mu$ and $M \subset H$ be a closed subspace. Secondly, we have the Hilbert orthogonal decomposition. For any $x\in H$, $x$ can be decomposed as $x=y+z,y\in M$, uniquely. Let $P$ be a projection operator such that $y=Px$ is the projection of $x$ on $M$. Finally, we limit the measure to $M$. For any Borel set $G_{M}$ of $M$, we have $\mu_{M}(G_{M}):=\mu(\{x\in H\mid Px\in G_{M}\})$. Obviously, $\mu_{M}$ is a probabilistic measure on $M$. Then we have the following definition of the Gel'fand width in probabilistic setting.

\begin{Definition}[\cite{ref10}]\label{def2}
    Let $X$ and $F^n$ be consistent with $d^n(W,X)$. Let $H$ be a Hilbert space and $H$ can be continuously embedded into $X$. Meanwhile, $\mu$ and $G_{\delta}$ are defined in the same way as in Definition \ref{def1}. Then, for any $\delta\in \left( 0,1 \right]$, the probabilistic Gel'fand $(n,\delta)$-width of $H$ in $X$ is given by
    $$d_{\delta}^{n}(H,\mu,X):=\inf_{G_{\delta}} d^{n} (H \setminus G_{\delta},X)=\inf_{G_{\delta}}\inf_{F^{n}}\sup_{x\in (H\setminus G_{\delta})\cap F^{n}}\|x\|.$$
    Meanwhile, $G_{\delta}$ satisfies that for any closed subspace $M$ of $H$, 
    \begin{equation}\label{eq1.1}
        \mu_{M}(G_{\delta}\cap M)\leq\delta,\tag{1.1}
    \end{equation}
    which guarantees $(H\setminus G_{\delta})\cap F^{n}$ has enough elements.
\end{Definition}

\begin{Remark}\label{rmk2}
    According to the definition of $d^{n}(W,X)$, we have at least $\{0\}\subset (H\setminus G_{\delta})\cap F^{n}$, so $(H\setminus G_{\delta})\cap F^{n} \neq \varnothing$, which makes sure that the probabilistic Gel'fand width makes sense. Namely, $G_{\delta}$ satisfies $\{0\}\not\subset G_{\delta}$, and when $\delta=1$, $G_\delta$ cannot equal to $H$ and $H\setminus G_\delta$ is a non-empty set of measure zero. Otherwise, $G_\delta=G_1=H$ always leads to $\mu_{M}(G_\delta\cap M)=1$ and $H\setminus G_{\delta}=\varnothing$. In this way, condition \ref{eq1.1} cannot guarantee $(H\setminus G_{\delta})\cap F^{n}$ has enough elements. Therefore, our condition $\{0\}\subset (H\setminus G_{\delta})\cap F^{n}$ or $\{0\}\not\subset G_{\delta}$ is necessary. In fact, if $H$ can be continuously embedded into $X$, $H$ can be regarded as a subspace $W$ of $X$. So $d_\delta^n(H,\mu,X)=d_\delta^n(W,\mu,X)$.
\end{Remark}

\begin{Definition}[\cite{ref11}]\label{def3}
    Let $W$, $X$ and $F^n$ be the same as $d^n(W,X)$, $\mu$ be the same as Definition \ref{def1}. The $p$-average Gel'fand $n$-width of $W$ in $X$ is given by
    $$d_{(a)}^{n}(W,\mu,X)_{p}:=\inf\limits_{F_{n}}\left(\int_{W\cap F^{n}}\|x\|^{p}d\mu(x)\right)^{1/p},\ 0<p<\infty.$$
\end{Definition}

Here, we provide clarifications for relevant notations. Let $c_{j}>0,\ j=0,1,2,\cdots$ represent positive constants, and let $a(y)$ and $b(y)$ be two arbitrary positive functions defined on the set $D$. $a(y)\gg b(y)$ or $a(y)\ll b(y)$ means that for all $y\in D,\ a(y)\geq c_{1}b(y)$ or $a(y)\leq c_{2}b(y)$, respectively. And we use $a(y)\asymp b(y)$ to express that $a(y)\gg b(y) \ \text{and}\  a(y)\ll b(y),\ y\in D$.

\section{The generalization of the probabilistic Gel'fand width}\label{sec2}

In this section, we will use the Birkhoff orthogonality to generalize the definition of the probabilistic Gel'fand width from the Hilbert space to the strictly convex reflexive space. According to the guiding process of Definition \ref{def2}, two very important definitions are mentioned: orthogonal decomposition and projection operator. If we want to generalize the probabilistic Gel'fand width of the Hilbert space, we need to consider projection first to decompose vectors. At the same time, orthogonality should also be considered to satisfy orthogonal decomposition and geometric intuition. 

Here, we give some definitions and lemmas in function approximation theory and geometry. These lemmas will greatly simplify our proof of theorem \ref{thm2.1}.

\begin{Definition}\label{def4}
    Let $(X,\|\cdot\|)$ be a normed linear space and $F \subset X$. Consider selecting elements from $F$ in some way as approximate representations of the elements in $X$. Then for any $x\in X$ and $F \neq \varnothing$,
    $$\mathscr{L}_F(x)=\{ u_0 \in F \ | \ \Vert x-u_0 \Vert = \inf\limits_{u \in F} \| x-u \|\}$$
    is called the set of the best approximation elements.
    \begin{enumerate}
        \item[(1)] If for each $x\in X$, we have $\mathscr{L}_F(x) \neq \varnothing$, then $F$ is called an existential set in $X$.
        \item[(2)]If for each $x\in X$, $\mathscr{L}_F(x)$ is an empty set or a one-point set, then $F$ is called a unique set in $X$.
        \item[(3)]If for each $x\in X$, $\mathscr{L}_F(x)$ is a one-point set, then $F$ is called a Chebyshev set in $X$. Namely, a Chebyshev set is both a unique set and an existential set.
    \end{enumerate}
\end{Definition}

\begin{Definition}\label{def5}
    When $F$ is an existential set, take any $x\in X$, and we can always find an element $y$ in $\mathscr{L}_F(x)$ as the image of $x$. This kind of mapping denoted by $P_F$ is called metric projection or the best approximation operator from $X$ to $F$.
\end{Definition}

\begin{Remark}\label{rmk3}
    If $F$ is an existential set, we have $P_F(x)=\mathscr{L}_F(x)=\{ u_0 \in F \ | \ \Vert x-u_0 \Vert = \inf\limits_{u \in F} \| x-u \|\}$. Otherwise, $P_F$ does not exist.
\end{Remark}

\begin{Definition}[\cite{ref15}]\label{def6}
    Let $X$ be a Banach space and $x,y \in X$. Then $x$ is said to be Birkhoff orthogonal to $y$, written as $x \perp_B y$, if for any scalar $\lambda$, $\|x+\lambda y\|\geq\|x\|$.

    Let $Y\neq \varnothing$ be a linear subspace in $X$. If for each $y\in Y$ and any scalar $\lambda$, we have $\|x+\lambda y\|\geq\|x\|$, then $x$ is said to be Birkhoff orthogonal to $Y$, written as $x \perp_B Y$.
\end{Definition}

\begin{Remark}\label{rmk4}
    The Birkhoff orthogonality does not satisfy symmetry. 
\end{Remark}
    
\begin{Remark}\label{rmk5}  
    In function approximation theory, we have the following variational condition: Let $Y\neq \varnothing$ be a linear subspace in $X$. For $x,x_0\in X$ and $e=x-x_0$, it is easy to know that $x_0\in\mathscr{L}_Y(x)\Leftrightarrow \|e+y\|\geq\|e\|,\forall y\in Y$. It can be described in geometric language as orthogonality in a general normed space: for each $y\in Y$, we have $\|x+ y\|\geq\|x\|$, then $x$ is said to be orthogonal to $Y$. Since $Y$ is a linear subspace, for any $y\in Y$ and any scalar $\lambda$, $\lambda y\in Y$. So, the definition of orthogonality in function approximation theory is the same as Birkhoff orthogonality. They all embody the idea that the vertical distance is the shortest.
\end{Remark}

\begin{lem}\label{lem2.1}
    Any closed linear subspace in a reflexive Banach space is an existential set.
\end{lem}

\begin{lem}\label{lem2.2}
    The convex set in the strictly convex space is a unique set.
\end{lem}

\begin{lem}[\cite{ref16}]\label{lem2.3}
    A normed linear space is strictly convex if and only if $\|x\|+\|y\|=\|x+y\|$ and $y\neq 0$ imply the existence of a number $t$ for which $x=ty$.
\end{lem}

we give the following decomposition theorem based on the Birkhoff orthogonality.

\begin{thm}\label{thm2.1}
    Let $X$ be a strictly convex reflexive space, $M$ be a closed linear subspace in $X$. For any $x\in X$, there exists a unique decomposition $x=m+n$, where $m \in M$ is the best approximation element of $x$ under the metric projection and $n\perp_B M,n\in M^{\perp_B}$. $M^{\perp_B}$ is the left orthogonal complement subspace of $M$ in $X$. Meanwhile, we have $X=M\oplus M^{\perp_B}$.
\end{thm}

\begin{proof}
    It is easy to know that a reflexive space must be a Banach space. Since $M$ is a closed linear subspace and any linear subspace must be a convex set, $M$ is a convex set. According to Lemma \ref{lem2.1} and Lemma \ref{lem2.2}, we have $M$ is a Chebyshev set. There must exist a unique $m$ such that $P_F(x)=\{m\}$. Obviously, $n$ is also unique. So we obtain the unique decomposition of $x$.

    For $n=x-m$ and number field $\mathbb{K}$, assume $\exists \lambda\in \mathbb{K}$ and $m'\in M$, such that $\|n+\lambda m'\|< \|n\|$. We have
    \begin{align*}
        \|n+\lambda m'\|&=\|x-m+\lambda m'\|\\ &=\|x-(m-\lambda m')\|\\ &<\|n\|=\|x-m\|=\inf_{u\in M}\|x-u\|.
    \end{align*}
    Since $M$ is a linear subspace and $m,m'\in M$, we have $m-\lambda m'\in M$. We obtain a contradiction. Therefore, $\forall \lambda\in \mathbb{K}$ and $m'\in M$, there must be $\|n+\lambda m'\|\geq \|n\|$. Therefore, we obtain $n\perp_B M$.

    Let $M^{\perp_B}=\{u\in X \mid u\perp_B M\}$. It is an orthogonal complement of $M$ based on Birkhoff orthogonality. We will prove that it is a subspace. It is obvious to maintain scalar multiplication. For addition operations, we suppose that for any scalar $\lambda$ and $n_1,n_2\in M^{\perp_B}$, we have $\|n_1+\lambda m\|\geq \|n_1\|$ and $\|n_2+\lambda m\|\geq \|n_2\|$. Since $x$ is a strictly convex normed linear space, according to Lemma \ref{lem2.3}, we have 
    \begin{align*}
    \|n_1+n_2+\lambda m\| &=\|n_1+n_2\|+\|\lambda m\|\\ &=\|n_1+n_2\|+|\lambda|\|m\|\\ &\geq \|n_1+n_2\|.
    \end{align*}
    Therefore, $M^{\perp_B}$ is a subspace.

    Because $\forall x \in X, x=m+n,m\in M,n \in M^{\perp_B}$ and this decomposition is unique, it is obvious that $X=M\oplus M^{\perp_B}$.
\end{proof}

Then, we can define the probabilistic Gel’fand width of the strictly convex reflexive space by imitating the way this width of Hilbert space is defined.

\begin{Definition}\label{def7}
    Let $X$, $F^n$, $\mu$, $\delta$ and $G_{\delta}$ be consistent with Definition \ref{def2}. Let $W$ be a strictly convex reflexive space and $W$ can be continuously embedded into $X$. Then the probabilistic Gel'fand $(n,\delta)$-width of $W$ in $X$ is given by
    $$d_{\delta}^{n}(W,\mu,X):=\inf_{G_{\delta}} d^{n} (W \setminus G_{\delta},X)=\inf_{G_{\delta}}\inf_{F^{n}}\sup_{x\in (W\setminus G_{\delta})\cap F^{n}}\|x\|,$$
    where $G_{\delta}$ satisfies that for any closed subspace $M$ of $W$, $\mu_{M}(G_{\delta}\cap M)\leq\delta$. Meanwhile, according to the definition of Gel'fand width in the classical setting, $\{0\}\subset (W\setminus G_{\delta})$. These two conditions guarantee the set, $(W\setminus G_{\delta})\cap F^{n}\neq \varnothing$, has enough elements.
\end{Definition}

\section{The relationship between probabilistic Gel'fand width and probabilistic linear width of the Hilbert space}

In Section \ref{sec2}, we have generalized the definition of the probabilistic Gel'fand width. Naturally, we want to study the relationship between probabilistic Gel'fand width and probabilistic linear width. Based on the better and natural definition of Gel'fand width in the classical setting chosen before, a loophole in the definition of probabilistic Gel'fand width is filled. Then, we can safely discuss the relationship we aim to study. In this section, we improve the proof of Ref\cite{ref11} and then give the equality relation between the probabilistic Gel'fand width and the probabilistic linear width of the Hilbert space. Considering the metric projection $P_F$ in the normed space does not have linearity, while the projection $P$ in Hilbert space is linear, we can only get the relationship in Hilbert space now. The following lemmas will play an important role in the proof of theorem \ref{thm3.1}.

\begin{lem}[\cite{ref10}]\label{lem3.1}
    Suppose $H$ is a Hilbert space, $(X,\|\cdot\|)$ is a normed linear space, $H$ can be continuously embedded into $X$, $\delta\in(0,1]$, $\mu$ is a probabilistic measure on $H$. Then $$\lambda_{n,\delta}(H,\mu,X)\leq d^n_\delta(H,\mu,X).$$
\end{lem}

\begin{lem}[\cite{ref14}]\label{lem3.2}
    Let $(X,\|\cdot\|)$ be a normed linear space and $W$ a subset of $X$, then $$\lambda_n(W,X)\geq d^n(W,X).$$
\end{lem}

\begin{thm}\label{thm3.1}
    Let $(X,\|\cdot\|)$ be a normed linear space and $H$ be a Hilbert space. $H$ can be continuously embedded into $X$ and $\delta\in (0,1]$. Then
    $$\lambda_{n,\delta}(H,\mu,X)= d^n_\delta(H,\mu,X).$$
\end{thm}

\begin{proof}
     On the one hand, according to Lemma \ref{lem3.1}, we have $\lambda_{n,\delta}(H,\mu,X)\leq d^n_\delta(H,\mu,X)$, $\delta\in (0,1]$. 
     
     On the other hand, for any $\delta\in (0,1]$, according to Lemma \ref{lem3.2}, 
     \begin{align*}
         \lambda_{n,\delta}(H,\mu,X)&= \inf_{G_\delta} \lambda_n (H\setminus G_\delta, X)\\ &\geq \inf_{G_\delta} d^n(H\setminus G_\delta, X)\\ &= d^n_\delta(H,\mu,X).
     \end{align*}
     
     Therefore, $$\lambda_{n,\delta}(H,\mu,X) = d^n_\delta(H,\mu,X),\delta\in (0,1].$$
\end{proof}

\section{Gel’fand widths of Sobolev spaces}


\subsection{Main results}

First, let us give some introductions to $W_{2}^{r}(\mathbb{T})$, $MW_{2}^{r}(\mathbb{T}^d)$ and $S_\infty$.

Let $d\text{-dimensional}$ torus $\mathbb{T}^d=[0,2\pi]^d,\ d\in \mathbb{Z}_+$. When $d=1$, let $\mathbb{T}:=\mathbb{T}^1$. We consider the Hilbert space $L_2(\mathbb{T}^d)$ of all $2\pi\text{-periodic}$ functions $x(t)$ where $t=(t_1,\cdots,t_d)\in \mathbb{T}^d$. Let $k=(k_1,\cdots,k_d)\in \mathbb{Z}^{d}$ and $(k,t)=\sum_{j=1}^dk_jt_j$. The Fourier series of $x(t)$ is as follows:
$$x(t)=\sum\limits_{k\in\mathbb{Z}^{d}} c_{k} e_{k}(t))=\sum\limits_{k\in\mathbb{Z}^{d}} \hat{x}(k) e_{k}(t),$$
where $e_{k}(t):=\exp(i(k,t))$, Fourier coefficient $\displaystyle c_{k}=\hat{x}(k)=\frac{1}{(2\pi)^d} \int_{\mathbb{T}^d}x(t)\exp(-i(k,t))\mathrm{d}t$. 

The inner product defined on $L_2(\mathbb{T}^d)$ is
$$\langle x,y \rangle =\frac{1}{(2\pi)^{d}}\int_{T^{d}}x(t)\overline{y(t)} \mathrm{d}t, \quad x,y\in L_{2}(\mathbb{T}^{d}).$$

For any $ r=(r_{1},\cdots,r_{d})\in \mathbb{R}^{d}$, Weyl $r\text{-fractional}$ derivative is given by
$$ x^{(r)}(t):=(D^{r}x)(t)=\sum_{k\in\mathbb{Z}^{d}}(ik)^{r}c_{k}e_{k}(t), $$
where $(ik)^{r}=\displaystyle \prod\limits_{j=1}^{d}|k_{j}|^{r_{j}}\exp(\frac{\pi i}{2} \text{sgn} r_{j}).$

For $r\in \mathbb{R}_+$, the univariate Sobolev space is denoted by
$$W_{2}^{r}(\mathbb{T})=\{x(t)\in L_{2}(\mathbb{T})\mid x^{(r)}(t)\in L_{2}(\mathbb{T}),\hat{x}(0)=0\}.$$
For any $|\alpha|\leq r$, it is obvious that $x^{(\alpha)}(t)\in L_{2}(\mathbb{T})$. The norm defined on $W_{2}^{r}(\mathbb{T})$ is $\|x\|_{W_2^r}:=\langle x^{(r)},x^{(r)}\rangle^\frac{1}{2}$. Therefore, $W_{2}^{r}(\mathbb{T})$ is a Hilbert space with inner product
$$ \langle x,y \rangle_1:=\langle x^{(r)},y^{(r)}\rangle. $$

Equip $W_{2}^{r}(\mathbb{T})$ with a Gaussian measure $\mu$. The mean of $\mu$ is $0$, and the correlation operator $C_{\mu}$ has eigenfunctions $e_{k}=\exp(ik(\cdot))$ and eigenvalues $\lambda_{k}=|k|^{-\rho}$, $k\in \mathbb{Z}_0$, $\mathbb{Z}_0:=\mathbb{Z}\setminus\{0\}$, $\rho>1$. Then, $C_{\mu}e_{k}=\lambda_{k}e_{k}$. 

Let $\mathscr{B}$ be any Borel subset in $\mathbb{R}^{n}$. For any orthonormal system $y_{1},y_{2},\cdots,y_{n}$ in $L_{2}(\mathbb{T})$, the cylindrical subsets in $W_{2}^{r}(\mathbb{T})$ is denoted by
$$ S=\{x\in W_{2}^{r}(\mathbb{T}) \mid(\langle x,y_{1}^{(-r)}\rangle_1,\cdots,\langle x,y_{n}^{(-r)}\rangle_1)\in\mathscr{B}\},$$

Let $\sigma_{j}=\langle C_{\mu}y_{j},y_{j}\rangle,\ j=1,\cdots,n$, the Gaussian measure $\mu$ on $S$ is defined by
$$\mu(S):=\prod_{j=1}^{n}(2\pi\sigma_{j})^{-\frac{1}{2}}\int_{\mathscr{B}}\exp\left(-\sum\limits_{j=1}^{n}\frac{|u_{j}|^{2}}{2\sigma_{j}}\right) \mathrm{d}u_{1} \cdots \mathrm{d}u_{n}.$$

Then, for $r=(r_{1},\cdots,r_{d})\in \mathbb{R}_{+}^{d}$, the multivariate Sobolev space with mixed derivative $MW_{2}^{r}(\mathbb{T}^{d})$ is denoted by 
$$ MW_{2}^{r}(\mathbb{T}^{d})=\{x\in L_{2}(\mathbb{T}^{d}) \mid x^{(r)}\in L_{2}(\mathbb{T}^{d}),\int_{0}^{2\pi}x(t)\mathrm{d}t_{j}=0,j=1,\cdots,d\}. $$
The integral condition in the set represents that if $k_{1} k_{2}\cdots k_{d}=0$, $c_{k}=0$.

The norm defined on $MW_{2}^{r}(\mathbb{T}^{d})$ is $\|x\|_{MW_2^r}^2:=\langle x^{(r)},x^{(r)}\rangle$. Therefore, $MW_{2}^{r}(\mathbb{T}^{d})$ is a Hilbert space with inner product
$$ \langle x,y \rangle_d:=\langle x^{(r)},y^{(r)}\rangle. $$

Let $\mathbb{Z}_{0}^{d}=\{k=(k_{1},\cdots,k_{d})\in\mathbb{Z}^{d} \mid k_{i}\neq0,i=1,\cdots,d\}$ and $\rho>1$. Equip $MW_{2}^{r}(\mathbb{T}^{d})$ with a Gaussian measure $\mu$. For the convenience of writing, we continue to use $\mu$. However, $\mu$ here is not the same as in the univariate Sobolev space. The mean of $\mu$ is $0$, and the correlation operator $C_{\mu}$ has eigenfunctions $e_{k}=\exp(i(k,\cdot))$ and eigenvalues $\lambda_{k}=|k|^{-\rho}$, $k\in \mathbb{Z}_0^d$. Then, $C_{\mu}e_{k}=\lambda_{k}e_{k}$. 

For any orthonormal system $s_{1},s_{2},\cdots,s_{n}$ in $L_{2}(\mathbb{T}^{d})$, the cylindrical subsets in $MW_{2}^{r}(\mathbb{T}^{d})$ is denoted by
$$ MS=\{x\in MW_{2}^{r}(\mathbb{T}^{d}) \mid(\langle x,s_{1}^{(-r)}\rangle_d,\cdots,\langle x,s_{n}^{(-r)}\rangle_d)\in\mathscr{B}\},$$

Let $\sigma_{j}=\langle C_{\mu}s_{j},s_{j}\rangle,\ s=1,\cdots,n$, the Gaussian measure $\mu$ on $MS$ is defined by
$$ \mu(MS):=\prod_{j=1}^{n}(2\pi\sigma_{j})^{-\frac{1}{2}}\int_{\mathscr{B}}\exp\left(-\sum\limits_{j=1}^{n}\frac{|u_{j}|^{2}}{2\sigma_{j}}\right)\mathrm{d}u_{1}\cdots \mathrm{d}u_{n}.$$

Let $l_q,1\leq q<\infty$ be the classical sequence space and $l_\infty$ be the normed linear space consisting of all bounded sequences endowed with the norm:
$$\|x\|_{l_\infty}=\sup_{j\in\mathbb{Z}}|x_j|.$$

The definition of $S_q(\mathbb{T}^d),1\leq q\leq \infty$ is given by
\[ S_q(\mathbb{T}^d) := \{ x\in L_1(\mathbb{T}^d)\mid \{\hat{x}(k)\}_{k\in \mathbb{Z}^d}\in l_q,\|x\|_{S_q}:=\|\{\hat{x}(k)\}\|_{l_q} \}. \]
When $d=1$, we get $S_q(\mathbb{T}),1\leq q\leq \infty$.

Obviously, $S_q(1\leq q\leq \infty)$ is the normed linear subspace of $L_1$. According to Parseval equalities, if $2< q\leq\infty$, $S_{q}(\mathbb{T}^{d})\supset L_{q}(\mathbb{T}^{d})$. Therefore, according to H\"{o}lder inequalities, when $2< q\leq\infty$ and $r >\max\{0,\frac{1}{2}-\frac{1}{q}\}$, $MW_{2}^{r}(\mathbb{T}^{d})$ can be continuously embedded into the space $S_{q}(\mathbb{T}^{d})$. Here, $r >\max\{0,\frac{1}{2}-\frac{1}{q}\}$ means $r=(r_1,\cdots,r_d)$ satisfies $r_i>\max\{0,\frac{1}{2}-\frac{1}{q}\},i=1,\cdots,d$. In the same way, when $r >\max\{0,\frac{1}{2}-\frac{1}{q}\}$, $W_{2}^{r}(\mathbb{T})$ can be continuously embedded into the space $S_{q}(\mathbb{T})$.

In this paper, we only need to consider $q=\infty$ and we always assume $r>\frac{1}{2}$. 

Xu \cite{ref12} has studied the probabilistic linear width of $W_{2}^{r}(\mathbb{T})$ in $S_q(\mathbb{T})$ and Wang \cite{ref17} has studied the probabilistic linear width of $MW_{2}^{r}(\mathbb{T}^d)$ in $S_q(\mathbb{T}^d)$. According to our relationship between probabilistic Gel'fand width and probabilistic linear width, the following lemmas will be used to give the probabilistic Gel'fand widths.

\begin{lem}[\cite{ref12}]\label{lem4.1}
    Let $1\leq q \leq \infty$, $\rho>1$, $r>\frac{1}{2}$, and $\delta \in \left( 0,\frac{1}{2} \right]$. Then, the probabilistic linear $(n,\delta)\text{-widths}$ of $W_{2}^{r}(\mathbb{T})$ with Gaussian measure $\mu$ in the metric of $S_q(\mathbb{T})$ satisfy the following asymptotic relation:
    $$\lambda_{n,\delta}(W_{2}^{r}(\mathbb{T}),\mu,S_q(\mathbb{T})) \asymp 
    \begin{cases} 
    n^{-(r+\frac{\rho}{2}-\frac{1}{q})} \sqrt{1+\frac{1}{n}\ln\frac{1}{\delta}},&1\leq q\leq 2,\\
    n^{-(r+\frac{\rho}{2}-\frac{1}{q})} \left(1+n^{-\frac{1}{q}}\sqrt{\ln\frac{1}{\delta}}\right),&2<q<\infty,\\
    n^{-(r+\frac{\rho}{2})}\sqrt{\ln \frac{n}{\delta}},& q=\infty.
    \end{cases}$$
\end{lem}

\begin{lem}[\cite{ref17}]\label{lem4.2}
    Let $1\leq q \leq \infty$, $\rho>1$, $r=(r_{1},\cdots, r_{d})\in \mathbb{R}^{d}$, $\frac{1}{2}<r_{1}=\cdots=r_{v}\leq r_{v+1}\leq\cdots\leq r_{d}$, and $\delta \in \left( 0,\frac{1}{2} \right]$. Then, the probabilistic linear $(n,\delta)\text{-widths}$ of $MW_{2}^{r}(\mathbb{T}^d)$ with Gaussian measure $\mu$ in the metric of $S_q(\mathbb{T}^d)$ satisfy the following asymptotic relation:
    \begin{enumerate}
        \item[(1)] $ 1\leq q\leq 2, $
        $$ \lambda_{n,\delta}(MW_{2}^{r}(\mathbb{T}^{d}),\mu,S_q(\mathbb{T}^d)) \asymp (n^{-1}\ln^{v-1}n)^{r_{1}+\frac{\rho}{2}-\frac{1}{q}} (\ln^{\frac{v-1}{q}}{n}) \sqrt{1+\frac{1}{n}\ln{\frac{1}{\delta}}}. $$
        \item[(2)] $ 2<q<\infty, $
        $$ \lambda_{n,\delta}(MW_{2}^{r}(\mathbb{T}^{d}),\mu,S_q(\mathbb{T}^d)) \asymp (n^{-1}\ln^{v-1}n)^{r_{1}+\frac{\rho}{2}-\frac{1}{q}} (\ln^{\frac{v-1}{q}}{n}) \left(1+n^{-\frac{1}{q}} \sqrt{\ln{\frac{1}{\delta}}}\right). $$
        \item[(3)] $ q=\infty, $
        $$ \lambda_{n,\delta}(MW_{2}^{r}(\mathbb{T}^{d}),\mu,S_q(\mathbb{T}^d)) \asymp (n^{-1}\ln^{v-1}n)^{r_{1}+\frac{\rho}{2}}\sqrt{\ln\frac{n}{\delta}}. $$
    \end{enumerate}
\end{lem}

Through the relationship we gave before and the lemmas mentioned before, we can easily obtain the probabilistic Gel'fand width. Our main results are presented as follows:

\begin{thm}\label{thm4.1}
Let $\rho>1$, $r>\frac{1}{2}$, and $\delta \in \left( 0,\frac{1}{2} \right]$, $0<p<\infty$. Then, as $n\rightarrow \infty$, the Gel'fand widths of the univariate Sobolev space $W_{2}^{r}(\mathbb{T})$ equipped with Gaussian measure $\mu$ in $S_\infty(\mathbb{T})$ satisfy the following asymptotic equalities:
    \begin{enumerate}
    \item[(1)] Probabilistic Gel'fand width
        $$ d_{\delta}^{n}(W_{2}^{r}(\mathbb{T}),\mu,S_\infty(\mathbb{T})) \asymp n^{-(r+\frac{\rho}{2})}\sqrt{\ln\frac{n}{\delta}}. $$
    \item[(2)] Average Gel'fand width
        $$ d_{(a)}^{n}(W_{2}^{r}(\mathbb{T}),\mu,S_\infty(\mathbb{T}))_p \asymp n^{-(r+\frac{\rho}{2})}\sqrt{\ln n}. $$
    \end{enumerate}
\end{thm}

\begin{thm}\label{thm4.2}
Let $\rho>1$, $r=(r_{1},\cdots, r_{d})\in \mathbb{R}^{d}$, $\frac{1}{2}<r_{1}=\cdots=r_{v}\leq r_{v+1}\leq\cdots\leq r_{d}$, and $\delta \in \left( 0,\frac{1}{2} \right]$, $0<p<\infty$. Then, as $n\rightarrow \infty$, the Gel'fand widths of the multivariate Sobolev space $MW_{2}^{r}(\mathbb{T}^d)$ equipped with Gaussian measure $\mu$ in $S_\infty(\mathbb{T}^d)$ satisfy the following asymptotic equalities:
    \begin{enumerate}
    \item[(1)] Probabilistic Gel'fand width
        $$ d_{\delta}^{n}(MW_{2}^{r}(\mathbb{T}^{d}),\mu,S_\infty(\mathbb{T}^d)) \asymp (n^{-1}\ln^{v-1}n)^{r_{1}+\frac{\rho}{2}}\sqrt{\ln\frac{n}{\delta}}. $$
    \item[(2)] Average Gel'fand width
        $$ d_{(a)}^{n}(MW_{2}^{r}(\mathbb{T}^{d}),\mu,S_\infty(\mathbb{T}^d))_p \asymp (n^{-1}\ln^{v-1}n)^{r_{1}+\frac{\rho}{2}}\sqrt{\ln n}. $$
    \end{enumerate}
\end{thm}

Theorem \ref{thm4.1} and theorem \ref{thm4.2} will be proved in the next two subsections, respectively.

\subsection{Proof of Theorem \ref{thm4.1}}

\begin{proof}
    
    According to the relationship given in theorem \ref{thm3.1}, and lemma \ref{lem4.1}, we obtain the probabilistic Gel'fand width of $W_{2}^{r}(\mathbb{T})$:
    \begin{equation}\label{eq4.2.1}
        d_{\delta}^{n}(W_{2}^{r}(\mathbb{T}),\mu,S_\infty(\mathbb{T})) = \lambda_{n,\delta}(W_{2}^{r}(\mathbb{T}),\mu,S_\infty((\mathbb{T})) \asymp n^{-(r+\frac{\rho}{2})}\sqrt{\ln\frac{n}{\delta}}. \tag{4.2.1}
    \end{equation}
    
    We will calculate the asymptotic order from two aspects. First, we give the upper bound of the average Gel'fand width.

    According to the definition \ref{def7} of probabilistic Gel’fand width, there exists a linear subspace $F^n$ of $S_\infty(\mathbb{T})$ whose co-dimension is at most $n$, such that for any $\delta \in \left( 0,\frac{1}{2} \right]$ and some subset $G_\delta\subset W_{2}^{r}(\mathbb{T})$, $\mu(G_\delta)\leq \delta$, we have
    $$\sup_{x\in (W_{2}^{r}(\mathbb{T})\setminus G_\delta)\cap F^{n}}\|x\|_{S_\infty} \leq d_{\delta}^{n}(W_{2}^{r}(\mathbb{T}),\mu,S_\infty(\mathbb{T}))+\epsilon,\ \text{for any }\epsilon>0.$$
    Due to the arbitrariness of $\epsilon$, we have
    \begin{equation}\label{eq4.2.2}
        \sup_{x\in (W_{2}^{r}(\mathbb{T})\setminus G_\delta)\cap F^{n}}\|x\|_{S_\infty} \ll d_{\delta}^{n}(W_{2}^{r}(\mathbb{T}),\mu,S_\infty(\mathbb{T})). \tag{4.2.2}
    \end{equation}
    
    Let the set sequence $\{G_{2^{-k}}\}_{k=0}^{\infty}$ satisfy the condition that for any $k$, $\mu(G_{2^{-k}})\leq2^{-k}$ and $G_{1}=W_{2}^{r}(\mathbb{T})$. Therefore, $W_{2}^{r}(\mathbb{T})=\bigcup_{k=0}^\infty(G_2^{-k}\setminus G_2^{-k-1})$. 
    
    Then, according to inequality \ref{eq4.2.2}, we have
    \begin{align*}
	\left (d_{(a)}^{n}(W_{2}^{r}(\mathbb{T}),\mu,S_\infty(\mathbb{T}))_p \right)^{p}
        &\leq\int_{W_{2}^{r}(\mathbb{T})\cap F^{n}}\|x\|_{S_\infty}^{p} \mathrm{d}\mu \\
	&=\sum_{k=0}^{\infty}\int_{(G_{2^{-k}}\setminus G_{2^{-k-1}})\cap F^{n}} \|x\|_{S_\infty}^{p} \mathrm{d}\mu \\
	&\leq\sum_{k=0}^{\infty}\int_{(G_{2^{-k}}\setminus G_{2^{-k-1}})\cap F^{n}}\sup_{x\in(G_{2^{-k}}\setminus G_{2^{-k-1}})\cap F^{n}}\|x\|_{S_\infty}^{p}\mathrm{d}\mu\\
        &\leq\sum_{k=0}^{\infty}\int_{(G_{2^{-k}}\setminus G_{2^{-k-1}})\cap F^{n}}\sup_{x\in(W_{2}^{r}(\mathbb{T})\setminus G_{2^{-k-1}})\cap F^{n}}\|x\|_{S_\infty}^{p} \mathrm{d}\mu\\
	&\ll\sum_{k=0}^{\infty}\int_{(G_{2^{-k}}\setminus G_{2^{-k-1}})\cap F^{n}}d_{2^{-k-1}}^{n}\left(W_{2}^{r}(\mathbb{T}),\mu,S_\infty(\mathbb{T})\right)^{p} \mathrm{d}\mu.
    \end{align*}
    
    Apply equation \ref{eq4.2.1} to the inequality above:
     \begin{align*}
	\left (d_{(a)}^{n}(W_{2}^{r}(\mathbb{T}),\mu,S_\infty(\mathbb{T}))_p \right)^{p}
        &\ll\sum_{k=0}^{\infty}\int_{(G_{2^{-k}}\setminus G_{2^{-k-1}})\cap F^{n}}  \left( n^{-(r+\frac{\rho}{2})}\sqrt{\ln\frac{n}{2^{-k-1}}} \right)^p \mathrm{d}\mu\\
        &\ll\sum_{k=0}^{\infty}\int_{(G_{2^{-k}}\setminus G_{2^{-k-1}})\cap F^{n}}  \left(n^{-(r+\frac{\rho}{2})}\sqrt{(k+1)\ln2+\ln n}\right)^p \mathrm{d}\mu\\
        &\ll\sum_{k=0}^{\infty} \left(n^{-(r+\frac{\rho}{2})}\sqrt{(k+1)\ln2+\ln n}\right)^p \mu(G_{2^{-k}})\\
        &\ll n^{-(r+\frac{\rho}{2})p} \sum_{k=0}^{\infty} \left(\sqrt{(k+1)\ln2+\ln n}\right)^p 2^{-k}\\
        &\ll n^{-(r+\frac{\rho}{2})p} \sum_{k=0}^{\infty} \left( \left(\sqrt{\ln n}\right)^p+\left(\sqrt{(k+1)\ln2}\right)^p \right)2^{-k}.
    \end{align*}
    
    Since $\sum_{k=0}^{\infty} 2^{-k}(\sqrt{\ln n})^p$ and $\sum_{k=0}^{\infty} 2^{-k}(\sqrt{(k+1)\ln2})^p$ are convergent, we have 
    $$ \left (d_{(a)}^{n}(W_{2}^{r}(\mathbb{T}),\mu,S_\infty(\mathbb{T}))_p \right)^{p} \ll n^{-(r+\frac{\rho}{2})p}(\sqrt{\ln n})^p. $$
    Namely, 
    $$ (d_{(a)}^{n}(W_{2}^{r}(\mathbb{T}),\mu,S_\infty(\mathbb{T}))_p \ll n^{-(r+\frac{\rho}{2})}\sqrt{\ln n}. $$

    Next, we give the lower bound of the average Gel'fand width. According to equation \ref{eq4.2.1}, there exists a constant $c>0$, then
    \begin{equation}\label{eq4.2.3}
        d_{\frac{1}{2}}^{n}(W_{2}^{r}(\mathbb{T}),\mu,S_\infty(\mathbb{T})) > cn^{-(r+\frac{\rho}{2})}\sqrt{\ln2n}. \tag{4.2.3}
    \end{equation}
    
    Consider the set
    $$ G=\left\{ x\in W_{2}^{r}(\mathbb{T})\cap F^{n} \mid \|x\|_{S_\infty} > cn^{-(r+\frac{\rho}{2})}\sqrt{\ln2n} \right\},$$
    We claim $\mu(G)>\frac{1}{2}$. According to the definition of the Gel'fand width in the probabilistic setting, we obtain
    \begin{align*}
        d_{\frac{1}{2}}^{n}(W_{2}^{r}(\mathbb{T}),\mu,S_\infty(\mathbb{T})) &\leq \sup_{x\in(W_{2}^{r}(\mathbb{T})\setminus G)\cap F^{n}}\|x\|_{S_\infty}\\
	&\leq cn^{-(r+\frac{\rho}{2})}\sqrt{\ln2n}.
    \end{align*}
    It contradicts the inequality (\ref{eq4.2.3}). Then,
    \begin{align*}
	\int_{W_{2}^{r}(\mathbb{T})\cap F^{n}} \|x\|_{S_\infty}^{p} \mathrm{d}\mu &\gg \int_{G}\|x\|_{S_\infty}^{p} \mathrm{d}\mu\\
	&\gg \left( n^{-(r+\frac{\rho}{2})}\sqrt{\ln2n} \right)^{p}\mu(G)\\
	&\gg \left( n^{-(r+\frac{\rho}{2})}\sqrt{\ln n} \right)^{p}.
    \end{align*}
    Then,
    $$ d_{(a)}^{n}(W_{2}^{r}(\mathbb{T}),\mu,S_\infty(\mathbb{T}))_p \gg n^{-(r+\frac{\rho}{2})}\sqrt{\ln n}.$$

    So, we get $d_{(a)}^{n}(W_{2}^{r}(\mathbb{T}),\mu,S_\infty(\mathbb{T}))_p \asymp n^{-(r+\frac{\rho}{2})}\sqrt{\ln n}$ and finish the proof.
\end{proof}

\subsection{Proof of Theorem \ref{thm4.2}}

In this part, we discuss the Gel'fand widths of the multivariate Sobolev space $MW_{2}^{r}(\mathbb{T}^d)$ in $S_\infty(\mathbb{T}^d)$. For the convenience of writing, we will use $d_\delta^n$ instead of $d_\delta^{n}(MW_{2}^{r}(\mathbb{T}^{d}),\mu,S_\infty(\mathbb{T}^d))$ and $d_{(a)}^n$ instead of $d_{(a)}^{n}(MW_{2}^{r}(\mathbb{T}^{d}),\mu,S_\infty(\mathbb{T}^d))_p$ in this part.

\begin{proof}
     
     According to the relationship given in theorem \ref{thm3.1}, and lemma \ref{lem4.2}, we obtain the probabilistic Gel'fand width of $MW_{2}^{r}(\mathbb{T}^{d})$:
    \begin{equation}\label{eq4.3.1}
        d_{\delta}^{n} = \lambda_{n,\delta}(MW_{2}^{r}(\mathbb{T}^{d}),\mu,S_\infty(\mathbb{T}^d)) \asymp (n^{-1}\ln^{v-1}n)^{r_{1}+\frac{\rho}{2}}\sqrt{\ln\frac{n}{\delta}}. \tag{4.3.1}
    \end{equation}
    
    First, we give the upper bound of the average Gel'fand width.

    According to the definition \ref{def7} of probabilistic Gel’fand width, there exists a linear subspace $F^n$ of $S_\infty(\mathbb{T}^d)$ with co-dimension at most $n$, such that for any $\delta \in \left( 0,\frac{1}{2} \right]$ and some subset $G_\delta\subset MW_{2}^{r}(\mathbb{T}^d)$, $\mu(G_\delta)\leq \delta$, we have
    $$\sup_{x\in (MW_{2}^{r}(\mathbb{T}^d)\setminus G_\delta)\cap F^{n}}\|x\|_{S_\infty} \leq d_{\delta}^{n}(MW_{2}^{r}(\mathbb{T}^d),\mu,S_\infty(\mathbb{T}^d))+\epsilon,\ \text{for any }\epsilon>0.$$
    Due to the arbitrariness of $\epsilon$, we have
    \begin{equation}\label{eq4.3.2}
        \sup_{x\in (MW_{2}^{r}(\mathbb{T}^d)\setminus G_\delta)\cap F^{n}}\|x\|_{S_\infty} \ll d_{\delta}^{n}(MW_{2}^{r}(\mathbb{T}^d),\mu,S_\infty(\mathbb{T}^d)). \tag{4.3.2}
    \end{equation}
    
    Let the set sequence $\{G_{2^{-k}}\}_{k=0}^{\infty}$ satisfy the condition that for any $k$, $\mu(G_{2^{-k}})\leq2^{-k}$ and $G_{1}=MW_{2}^{r}(\mathbb{T}^d)$. Therefore, $MW_{2}^{r}(\mathbb{T}^d)=\bigcup_{k=0}^\infty(G_2^{-k}\setminus G_2^{-k-1})$. We have 
    
    \begin{align*}
	\left (d_{(a)}^{N}\right)^{p}
        &\leq\int_{MW_{2}^{r}(\mathbb{T}^{d})\cap F^{n}}\|x\|_{S_\infty}^{p}\mathrm{d}\mu\\
        &=\sum_{k=0}^{\infty}\int_{(G_{2^{-k}}\setminus G_{2^{-k-1}})\cap F^{n}}\|x\|_{S_\infty}^{p}\mathrm{d}\mu\\
        &\leq\sum_{k=0}^{\infty}\int_{(G_{2^{-k}}\setminus G_{2^{-k-1}})\cap F^{n}}\sup_{x\in(G_{2^{-k}}\setminus G_{2^{-k-1}})\cap F^{n}}\|x\|_{S_\infty}^{p}\mathrm{d}\mu\\
        &\leq\sum_{k=0}^{\infty}\int_{(G_{2^{-k}}\setminus G_{2^{-k-1}})\cap F^{n}}\sup_{x\in(W_{2}^{r}(\mathbb{T})\setminus G_{2^{-k-1}})\cap F^{n}}\|x\|_{S_\infty}^{p} \mathrm{d}\mu.
    \end{align*}
    Then, according to the inequality \ref{eq4.3.2}, we have
    \begin{align*}
        \left (d_{(a)}^{N}\right)^{p}
        &\ll\sum_{k=0}^{\infty}\int_{(G_{2^{-k}}\setminus G_{2^{-k-1}})\cap F^{n}}d_{2^{-k-1}}^{n}\left(MW_{2}^{r}(\mathbb{T}^{d}),\mu,S_\infty(\mathbb{T}^{d})\right)^{p} \mathrm{d}\mu.\\
        &\ll\sum_{k=0}^{\infty}d_{2^{-k-1}}^{n}\left(MW_{2}^{r}(\mathbb{T}^{d}),\mu,S_{\infty}(\mathbb{T}^{d})\right)^{p}\mu(G_{2^{-k}}).
    \end{align*}
    Meanwhile, according to equation \ref{eq4.3.1}, we have
    $$ d_{\delta}^{n} \ll (n^{-1}\ln^{v-1}n)^{r_{1}+\frac{\rho}{2}}\sqrt{\ln\frac{n}{\delta}}. $$
    Therefore, 
    \begin{align*}
        \left (d_{(a)}^{N}\right)^{p}
        &\ll \sum_{k=0}^{\infty} \left( (n^{-1}\ln^{v-1}n)^{r_{1}+\frac{\rho}{2}}\sqrt{\ln\frac{n}{2^{-k-1}}} \right)^p \mu(G_{2^{-k}})\\
        &\ll (n^{-1}\ln^{v-1}n)^{(r_{1}+\frac{\rho}{2})p} \sum_{k=0}^{\infty} \left( \sqrt{\ln{n} +(k+1)\ln{2}} \right)^p 2^{-k}\\
        &\ll (n^{-1}\ln^{v-1}n)^{(r_{1}+\frac{\rho}{2})p} \sum_{k=0}^{\infty} \left( \left(\sqrt{\ln{n}} \right)^p  + \left(\sqrt{(k+1)\ln{2}} \right)^p \right) 2^{-k}.
    \end{align*}
    In the same way as $W_2^r(\mathbb{T})$, we have
    $$ d_{(a)}^{n}(MW_{2}^{r}(\mathbb{T}^{d}),\mu,S_\infty(\mathbb{T}^d))_p \ll(n^{-1}\ln^{v-1}n)^{r_{1}+\frac{\rho}{2}}\sqrt{\ln n}. $$

    Next, we give the lower bound of the average Gel'fand width. According to equation \ref{eq4.3.1}, there exists a constant $c>0$, then
    \begin{equation}\label{eq4.3.3}
        d_{\frac{1}{2}}^{n}(MW_{2}^{r}(\mathbb{T}^d),\mu,S_\infty(\mathbb{T}^d)) > c(n^{-1}\ln^{v-1}n)^{r_{1}+\frac{\rho}{2}}\sqrt{\ln2n}. \tag{4.3.3}
    \end{equation}
    
    Consider the set
    $$ G'=\left\{ x\in MW_{2}^{r}(\mathbb{T}^d)\cap F^{n} \mid \|x\|_{S_\infty} > c(n^{-1}\ln^{v-1}n)^{r_{1}+\frac{\rho}{2}}\sqrt{\ln2n} \right\},$$
    We claim $\mu(G')>\frac{1}{2}$. The proof method is the same as $W_2^r(\mathbb{T})$. Otherwise, it will contradict the inequality (\ref{eq4.3.3}). Then,
    \begin{align*}
	\int_{MW_{2}^{r}(\mathbb{T}^d)\cap F^{n}} \|x\|_{S_\infty}^{p} \mathrm{d}\mu &\gg \int_{G'}\|x\|_{S_\infty}^{p} \mathrm{d}\mu\\
	&\gg \left( (n^{-1}\ln^{v-1}n)^{r_{1}+\frac{\rho}{2}}\sqrt{\ln2n} \right)^{p}\mu(G')\\
	&\gg \left( (n^{-1}\ln^{v-1}n)^{r_{1}+\frac{\rho}{2}}\sqrt{\ln2n} \right)^{p}.
    \end{align*}
    Namely,
    $$ d_{(a)}^{n}(MW_{2}^{r}(\mathbb{T}^{d}),\mu,S_\infty(\mathbb{T}^d))_p \gg (n^{-1}\ln^{v-1}n)^{r_{1}+\frac{\rho}{2}}\sqrt{\ln n}.$$
    
    So, we obtain 
    $$ d_{(a)}^{n}(MW_{2}^{r}(\mathbb{T}^{d}),\mu,S_\infty(\mathbb{T}^d))_p \asymp (n^{-1}\ln^{v-1}n)^{r_{1}+\frac{\rho}{2}}\sqrt{\ln n}. $$
\end{proof}

\section{Summary}

In this article, we mainly studied the generalization of the definition of the probabilistic Gel'fand width. By using the Birkhoff orthogonality, we generalize it from the Hilbert space to the strictly convex reflexive space. Then, we improve many details of previous works and then give the equality relation between the probabilistic Gel'fand width and the probabilistic linear width of the Hilbert space. In the end, we use the relation to obtain the Gel'fand widths of the univariate Sobolev space and the multivariate Sobolev space in the $S_{\infty}$-norm in the probabilistic settings and calculate their exact order of average Gel'fand widths. As we all know, with the development of computing science and internal intersections within mathematics, the complexity of algorithms and properties of functional spaces are more and more important. Our research in the width theory is an effective way to clearly articulate the essence of the problem.
 
However, our research still has some shortcomings. For example, since the metric projection operator does not possess the linear property, we still cannot extend the relationship between probabilistic Gel'fand width and probabilistic linear width of the Hilbert space to the strictly convex reflexive space. This is a question that we still need to study in the future.

\bibliographystyle{unsrt}
\bibliography{main} 

\begin{thebibliography}{10}

\bibitem{ref1}
J.~F. Traub, G.~W. Wasilkowski, and H.~Wozniakowski.
\newblock {\em Information-Based Complexity}.
\newblock Academic Press, New York, 1988.

\bibitem{ref2}
Tizian Wenzel, Daniel Winkle, Gabriele Santin, and Bernard Haasdonk.
\newblock Adaptive meshfree approximation for linear elliptic partial differential equations with pde-greedy kernel methods.
\newblock {\em BIT Numerical Mathematics}, 65(1):11, 2025.

\bibitem{ref3}
Vitaly Maiorov.
\newblock Average $n$-widths of the wiener space in the $l_\infty$-norm.
\newblock {\em Journal of Complexity}, 9:222--222, 1993.

\bibitem{ref4}
Vitaly Maiorov.
\newblock Kolmogorov's $(n\text{-}\delta)$-widths of the spaces of the smooth functions.
\newblock {\em Sbornik: Mathematics}, 79(2):265--279, 1994.

\bibitem{ref5}
Vitaly Maiorov.
\newblock Linear widths of function spaces equipped with the gaussian measure.
\newblock {\em Journal of Approximation Theory}, 77(1):74--88, 1994.

\bibitem{ref6}
Fang Gensun and Ye~Peixin.
\newblock Probabilistic and average linear widths of sobolev space with gaussian measure.
\newblock {\em Journal of Complexity}, 19(1):73--84, 2003.

\bibitem{ref7}
Fang Gensun and Ye~Peixin.
\newblock Probabilistic and average linear widths of sobolev space with gaussian measure in $l_\infty$-norm.
\newblock {\em Constructive approximation}, 20:159--172, 2003.

\bibitem{ref8}
Chen Guanggui and Fang Gensun.
\newblock Probabilistic and average widths of multivariate sobolev spaces with mixed derivative equipped with the gaussian measure.
\newblock {\em Journal of Complexity}, 20(6):858--875, 2004.

\bibitem{ref9}
Heping Wang.
\newblock Probabilistic and average linear widths of weighted sobolev spaces on the ball equipped with a gaussian measure.
\newblock {\em Journal of Approximation Theory}, 241:11--32, 2019.

\bibitem{ref10}
Xin Tan, Yanan Wang, Lu~Sun, Xingfeng Shao, and Guanggui Chen.
\newblock Gel’fand-n-width in probabilistic setting.
\newblock {\em Journal of Inequalities and Applications}, 2020(1):143, 2020.

\bibitem{ref11}
Yuqi Liu, Huan Li, and Xuehua Li.
\newblock Gel’fand widths of sobolev classes of functionsin the average setting.
\newblock {\em Annals of Functional Analysis}, 14(2):31, 2023.

\bibitem{ref12}
XU~Yanyan, CHEN Guanggui, GAN Ying, and XU~Yan.
\newblock Probabilistic and average linear widths of sobolev space with gaussian measure in space $s_q(t)(1 \leq q\leq \infty)$.
\newblock {\em Acta Mathematica Scientia}, 35(2):495--507, 2015.

\bibitem{ref13}
Ruihuan Wu, Yuqi Liu, and Huan Li.
\newblock Probabilistic and average gel’fand widths of sobolev space equipped with gaussian measure in the $s_q$-norm.
\newblock {\em Axioms}, 13(7):492, 2024.

\bibitem{ref14}
Allan Pinkus.
\newblock {\em N-widths in Approximation Theory}, volume~7.
\newblock Springer Science \& Business Media, 2012.

\bibitem{ref15}
Arpita Mal, Kallol Paul, and Debmalya Sain.
\newblock {\em Birkhoff-James Orthogonality and Geometry of Operator Spaces}.
\newblock Springer, 2024.

\bibitem{ref16}
Robert~C James.
\newblock Orthogonality and linear functionals in normed linear spaces.
\newblock {\em Transactions of the American Mathematical Society}, 61(2):265--292, 1947.

\bibitem{ref17}
Pei Wang.
\newblock Linear approximation characteristic of sobolev space with bounded mixed derivative in different computational setting.
\newblock Master's thesis, Xihua University, 2014.

\end{thebibliography}

\end{document}